\documentclass[10pt,a4paper]{amsart}
\usepackage[latin2]{inputenc}
\usepackage[T1]{fontenc}
\IfFileExists{lmodern.sty}{
\usepackage{lmodern}
}{}

\usepackage{euscript}
\usepackage{amsmath}
\usepackage{amsthm}
\usepackage{amssymb}
\usepackage{mathtools}
\mathtoolsset{mathic}

\usepackage{color,graphicx}
\IfFileExists{p2e-dvips.def}{
\usepackage{pict2e}
}{}

\usepackage{enumerate}

\usepackage[bookmarksnumbered,bookmarksopen,unicode]{hyperref}
\hypersetup{
  pdftitle={Surgery formula for Seiberg\textendash Witten invariants of
    negative definite plumbed 3-manifolds},
  pdfauthor={Gábor Braun and András Némethi},
  pdfkeywords={plumbed 3-manifolds, rational homology spheres,
    Seiberg\textendash Witten invariants, Casson invariant, local surface
    singularities, geometric genus, Neumann\textendash Wahl conjecture,
    Seiberg\textendash Witten invariant conjecture, surgery formula,
    splice-quotient singularities.},
  pdfsubject={MSC2000 Primary  57M27, 57R57; Secondary 32C35, 32S05,
    32S25, 32S45},
}

\numberwithin{equation}{section}
\numberwithin{equation}{subsection}

\theoremstyle{plain}
\newtheorem{theorem}[equation]{Theorem}
\newtheorem{lemma}[equation]{Lemma}
\newtheorem{proposition}[equation]{Proposition}
\newtheorem{corollary}[equation]{Corollary}
\newtheorem{definition}[equation]{Definition}
\newtheorem{conjecture}[equation]{Conjecture}

\theoremstyle{definition}

\newtheorem{remark}[equation]{Remark}

\DeclareMathOperator{\df}{def}
\DeclareMathOperator{\Hom}{Hom}
\DeclareMathOperator{\ssw}{sw}
\DeclareMathOperator{\supp}{supp}

\newcommand*{\pc}[1]{\operatorname{pc} #1}

\newcommand*{\constanttermatone}[1]{\operatorname{{coef}\sb{1}\sp{0}} #1}

\newcommand*{\Laurentpartat}[1]{\operatorname{L\sb{#1}}}
\DeclareMathOperator{\Pic}{Pic}
\newcommand{\Spinc}{\texorpdfstring{\ifmmode \operatorname{{Spin}\sp{c}} \else
  Spin\textsuperscript{c}\fi}{Spin-c}}

\newcommand{\et}{\EuScript{T}}

\newcommand{\cH}{\mathcal{H}}
\newcommand{\cL}{\mathcal{L}}
\newcommand{\cO}{\mathcal{O}}
\newcommand{\cV}{\mathcal{V}}

\newcommand{\setC}{\mathbb{C}}
\newcommand{\setQ}{\mathbb{Q}}

\newcommand{\setZ}{\mathbb{Z}}
\newcommand{\setN}{\mathbb{N}}

\providecommand{\coloneqq}{\mathrel{:=}}
\newcommand*{\mathid}[1]{\mathrm{#1}}

\begin{document}

\title[Surgery formula for the Seiberg\textendash Witten
invariants]{Surgery formula for Seiberg\textendash Witten invariants of
negative definite plumbed
  3-manifolds}

\author{Gábor Braun}
\thanks{The first author is partially supported by Hungarian National Research
  Fund, grants No.\ K~61007 and T~042769.}
\address{Rényi Institute of Mathematics\\
  Budapest\\
  Reáltanoda u. 13\textendash15\\
  1053\\
 Hungary}
\email{braung@renyi.hu}

\author{András Némethi}
\thanks{The second author is partially supported by OTKA grants.}
\address{Rényi Institute of Mathematics\\
  Budapest\\
  Reáltanoda u. 13\textendash15\\
  1053\\
 Hungary}
\email{nemethi@renyi.hu}



\begin{abstract}
  We derive a cut-and-paste surgery formula of Seiberg\textendash Witten invariants for
  negative definite plumbed rational homology \(3\)-spheres.  It is similar to
  (and motivated by) Okuma's recursion formula \cite[4.5]{Opg} targeting
  analytic invariants of splice-quotient singularities.
  Combining the two formulas
  automatically provides a proof of the equivariant version
  \cite[5.2(b)]{Line} of the Seiberg\textendash Witten invariant conjecture
  \cite{nemethi02._seiber_witten} for these singularities.
\end{abstract}

\maketitle

\section{Introduction}\label{Sec:1}

Problem 5 of the review article \cite{OSZ} of Ozsváth and Szabó is to develop
cut-and-paste techniques for calculating the Heegaard Floer homology of
\(3\)-manifolds.  In this article we obtain a possible answer 
at the level of the Seiberg\textendash Witten invariant (i.e.\ at the
level of the normalized Euler characteristic of the Heegaard Floer
homology): we provide the cut-and-paste surgery formula \eqref{eq:1} for the
Seiberg\textendash Witten invariants of plumbed rational homology \(3\)-spheres
associated with negative definite plumbing graphs.  In order to state it, we
fix some notations (for more details, see §\ref{Sec:2}).

For any graph \(G\), let \(\cV(G)\) denote its set of vertices.
Let \(\lvert S \rvert\) denote the size of the finite set \(S\).
Thus, \(\lvert \cV(G) \rvert\) is the number of vertices of \(G\).

Let $\Gamma$ be a connected plumbing graph.
Each vertex $w\in\cV(\Gamma)$ is
decorated by an integer $b_w$.  Let $\widetilde{X}(\Gamma)$ be the
\(4\)-manifold with boundary obtained by plumbing from \(\Gamma\), which 
we briefly recall.
The manifold \(\widetilde{X}(\Gamma)\) is a tubular neighbourhood of oriented \(2\)-spheres
\(E_w\) associated with the vertices \(w\) of the graph.
For every two adjacent vertices, their \(2\)-spheres  intersect transversally
at one point;
beside these, the \(2\)-spheres do not intersect each other.
The number \(b_w\) is the Euler number of the normal bundle of the
\(2\)-sphere of the vertex \(w\).

The manifold $\widetilde{X}(\Gamma)$ admits  a canonical \Spinc{} structure
$\widetilde{\sigma_{\mathid{can}}}$, see~\eqref{eq:4} for its characterization.

Set $\Sigma \coloneqq \partial \widetilde{X}(\Gamma)$.  We assume that $H_1(\Sigma; \setQ)=0$, or equivalently
that $\Gamma$ is a tree.

Set $L \coloneqq H_2(\widetilde{X}(\Gamma); \setZ)$ and \(L' \coloneqq
H^2(\widetilde{X}(\Gamma); \setZ)\).  These
groups are free with bases the classes \(E_w\) of the \(2\)-spheres and their
duals \(E^*_w\), respectively.

The graph \(\Gamma\) is \emph{negative definite} if the intersection form on \(L\)
is negative definite.  If this is the case then the canonical map \(L \to
L'\) is an embedding, which is an isomorphism over
$\setQ$, thus the intersection form extends to $L'$.  We shall write \((\cdot, \cdot)\)
for the intersection form and $x^2 \coloneqq (x,x)$ for any $x \in L'$.

For any \Spinc{} structure \(\sigma\), let \(c_1(\sigma) \in L'\) 
denote its first Chern class.

Finally, for any $\sigma\in \Spinc(\Sigma)$ and $v\in\cV(\Gamma)$, let $\cH_{\sigma,v}$ be the
rational function defined in (\ref{Hilbert}), which is a Weil-type twisted
zeta function.  We write $\cH^{\mathid{pol}}_{\sigma,v}$ for its \emph{polynomial part}
which is the unique polynomial for which \(\cH_{\sigma,v} - \cH^{\mathid{pol}}_{\sigma,v}\)
has negative degree (i.e.\ it is either \(0\) or the degree of the numerator
is less than the degree of the denominator).

\begin{theorem}
  \label{MTIntr}
  Let \(\Gamma\) be a connected negative definite plumbing graph of a rational
  homology \(3\)-sphere \(\Sigma\).  Let \(v\) be a vertex of \(\Gamma\), and let \(\Gamma_i\) be
  the components of \(\Gamma \setminus v\).  Let \(\widetilde{\sigma}\) be a \Spinc{} structure
  of \(\widetilde{X}(\Gamma)\) satisfying
  \begin{equation}
    \label{eq:2}
    -1 < \left( \frac{c_1(\widetilde{\sigma}) - c_1(\widetilde{\sigma_{\mathid{can}}})}{2}, E^*_v
    \right) \leq 0.
  \end{equation}
  Let \(\sigma\), \({\widetilde{\sigma}}_{i}\) and \(\sigma_i\) denote the restriction of
  \(\widetilde{\sigma}\) to \(\Sigma\), \(\widetilde{X}(\Gamma_i)\) and \(\Sigma_i \coloneqq \partial
  \widetilde{X}(\Gamma_i)\), respectively.  Then
  \begin{equation}\label{eq:1}
    \ssw_\sigma(\Sigma) + \frac{{c_1(\widetilde{\sigma})}^2 + \lvert \cV(\Gamma) \rvert}{8} =
    - \cH^{\mathid{pol}}_{\sigma,v}(1) + \sum_i \left( \ssw_{\sigma_i}(\Sigma_i) +
      \frac{{c_1(\widetilde{\sigma}_i)}^2 + \lvert \cV(\Gamma_i) \rvert}{8} \right).
  \end{equation}
\end{theorem}

\begin{remark}\label{OSZuj}
The \Spinc{} structure \(\sigma\) does not uniquely determine
 \(\widetilde{\sigma}\) and its restriction
 \({\widetilde{\sigma}}_{i}\)
via \eqref{eq:2}.  Nevertheless, the \Spinc{} structure  
$\sigma_i$  is independent of  the choice of  $\widetilde{\sigma}$; it 
depends only on $\sigma$. 
\end{remark}
\begin{remark}\label{OSZ2}
  Notice that this formula differs from those obtained from surgery exact
  triangles (of different versions) of Floer homologies (see e.g.\
  \cite{OSzP}): the surgery exact triangles involve three different
  \(3\)-manifolds, while our formula only connects the plumbed \(3\)-manifolds
  associated with $\Gamma$ and $\Gamma \setminus v$ (and another type of invariant, namely
  $\cH_{\sigma,v}$).  Moreover, in general, the surgery exact triangles mix several
  \Spinc{} structures (involving all the extensions $\widetilde{\sigma}$), while
  our formula involves only one extension $\widetilde{\sigma}$
  and one induced pair $(\widetilde{\sigma}_i,\sigma_i)$
  for any fixed $\sigma$.
\end{remark}

The proof uses the fact 
(see \cite[Theorem 2.4]{Nico5}, recalled here in \eqref{RT1})
that the Seiberg\textendash Witten invariant of $\Sigma$ is a linear
combination  of the
Reidemeister\textendash Turaev torsion $\et$ (\cite{Tu5}) and the 
Casson\textendash Walker invariant $\lambda$, together with explicit formulas for these
invariants.

In particular, the formula above is the consequence  of additivity formulas
for the invariants \({c_1(\widetilde{\sigma})}^2 + \lvert \cV(\Gamma) \rvert\),
$\et$ and $\lambda$, stated in \eqref{AK}, \eqref{AL} and \eqref{AT},
which are interesting for their own sake as well.

In §\ref{Sec:7} we exemplify (\ref{MTIntr}) for Seifert manifolds and
surgery manifolds $S^3_{-d}(K)$.  There we emphasize the arithmetical nature of
$\cH_{\sigma,v}$, too.

Any negative definite plumbed \(3\)-manifold appears as the link of a
complex surface singularity.  For some singularity links, the Taylor expansion
of $\cH_{\sigma,v}$ at the origin appears as the Hilbert (Poincaré) series of a
certain graded $\setC$-algebra.  In this way, $\cH^{\mathid{pol}}_{\sigma,v}(1)$
can be related with analytic invariants of the singularity.
For applications of (\ref{MTIntr}) in singularity theory, see
§\ref{Back} and (\ref{star}).

\section{Application in singularity theory.}
\label{Back}

\subsection{Seiberg\textendash Witten invariant conjecture}
\label{sec:seib-witt-conj}

 Let $(X,o)$ be an isolated
complex analytic normal surface singularity whose link
$\Sigma$ is a rational homology sphere.
Let $\pi\colon \widetilde{X}\to X$
be a good resolution with exceptional set $E$ (with irreducible components
${\{E_w\}}_w$), and $\Gamma$ its dual resolution graph (for details see e.g.\ 
\cite[§2.2]{INV}).  Then (the underlying $C^\infty$ manifold of)
$\widetilde{X}$ is the plumbed \(4\)-manifold
$\widetilde{X}(\Gamma)$ (for which in the sequel we will use all the
above notations).  The intersection form on $L$ is automatically negative
definite.

The group $L$ can also be regarded as the group of integral cycles (divisors)
of type $l = \sum_w m_w E_w$ in $\widetilde{X}$ with $m_w \in \setZ$.  As customary, we
denote by \(\cO_{\widetilde{X}}(l)\) the line bundle associated with \(l\).
This map $l\mapsto \cO_{\widetilde{X}}(l)$
extends uniquely to a group homomorphism $L'\to \Pic(\widetilde{X})$,
denoted similarly by $l'\mapsto \cO_{\widetilde{X}}(l')$, such that the Chern class
(multidegree) satisfies $c_1(\cO_{\widetilde{X}}(l'))=l'$ (see
\cite[3.4\textendash3.6]{Line}).

As usual, $h^1(\cL)$ denotes $\dim_\setC H^1(\widetilde{X},\cL)$.
In this way,
 the \emph{geometric genus} is
$p_g\coloneqq h^1(\cO_{\widetilde{X}})$.
 More generally, for the special set of
representatives 
\begin{equation*}
R\coloneqq  \left\{ \sum_wr_wE_w\in L': -1<
r_w\leq 0 \right\} \subset L'
\end{equation*} 
of the classes $L'/L$, we get the \emph{equivariant geometric genera}
${\{h^1(\cO_{\widetilde{X}}(l'))\}}_{l'\in R}$ of $(X,o)$ (the $L'/L=H_1(\Sigma; \setZ)$
eigen-decomposition of the geometric genus of the universal abelian cover of
$(X,o)$, see \cite[3.7]{Line} and \cite[2.2(3)]{Opg}).  They are subtle
analytic invariants of $(X,o)$, which guide crucial analytic aspects (e.g.\
equisingular deformations).  In general, they are not topological;
nevertheless, in \cite[5.2(b)]{Line}, the second author formulated essentially
the following conjecture, which predicts that in special cases, these
invariants can be recovered from the \emph{link} $\Sigma$:

\begin{conjecture}[Seiberg\textendash Witten invariant conjecture \cite{Line}]
\label{CONJ}
 Set $L_e$ for the \emph{effective}
integral cycles, i.e.\ $L_e\coloneqq \{\sum_w m_w E_w :
m_w\geq 0 \text{ for all $w$}\}$.  Set $R+L_e\coloneqq \bigcup_{l'\in
R}(l'+L_e)\subset L'$.

If the analytic structure of $(X,o)$ is `nice', then for all \(l' \in
  R+L_e\) one has
\begin{equation}\label{eq:3}
  -h^1 \left(\cO_{\widetilde{X}}(l') \right) = \ssw_{[l']*\sigma_{\mathid{can}}}(\Sigma)
  + \frac{{(c_1(\widetilde{\sigma_{\mathid{can}}}) + 2l')}^2 + \lvert \cV(\Gamma) \rvert
  }{8}.
\end{equation}
(For the definition of the \Spinc{} structure $[l']*\sigma_{\mathid{can}}$
of $\Sigma$, see §\ref{Char}.) 
\end{conjecture}

\begin{remark}\label{RIntr}\mbox{}
  \begin{enumerate}
  \item\label{RIntr3} It is part of the conjecture to clarify the meaning of
    `nice'.  In the
    original version \cite{Line,nemethi02._seiber_witten} the conjecture 
    was formulated for all
    $\setQ$-Gorenstein singularities, but counterexamples are given in \cite[§4]{SI}.
    On the other hand,  the conjecture holds for all rational singularities 
    (\cite{Line,Graded}, see also \cite{NOSZ}), and, in fact, here we shall
    prove it for all splice-quotient singularities, see
    (\ref{TH}).  Restricted to the case of the canonical \Spinc{} structure,
    it was verified for elliptic Gorenstein singularities (by combining
    \cite{Ninv} and \cite{NOSZ}), singularities with good $\setC^*$ action
    (\cite{nemethi04._seiber_witten}), and suspension hypersurface
    singularities defined by $f(x,y) + z^n = 0$
    with $f$ irreducible (\cite{nemethi05._seiber_witten}).  For a review of
    related  problems,
    see \cite{INV, Graded}.  For related results, see
    \cite{Co,CoS,CoS2,FMS,FS,NW,NWuj}.
  \item\label{item:1} As a byproduct of the main Theorem (\ref{MTIntr}),
    in Theorem (\ref{AN}) we provide a criterion which
    characterizes the singularities  satisfying \eqref{eq:3}.
  \item\label{RIntr1} The special case of the canonical \Spinc{} structure was
    conjectured in \cite{nemethi02._seiber_witten}.  It generalizes the Casson invariant
    conjecture of Neumann and Wahl formulated for any isolated complete
    intersection with \emph{integral} homology sphere link \cite{NW}.
  \item\label{RIntr2} In fact, \eqref{eq:3} essentially consists
    of (only) $\lvert H_1(\Sigma; \setZ) \rvert $ different identities.
    The reason is that the expression
    \begin{equation*}
      h^1 \left( \cO_{\widetilde{X}}(l') \right) 
      + \frac{{(c_1(\widetilde{\sigma_{\mathid{can}}}) + 2l')}^2 + \lvert \cV(\Gamma) \rvert
      }{8}.
    \end{equation*}
    depends only on $[l'] \in H_1(\Sigma; \setZ)$ for $l' \in R + L_e$
    by~\cite[5.3(c)]{Line}. 
    Therefore, it is enough to verify the identity~\eqref{eq:3}, say,
    for all $l' \in R$.
  \end{enumerate}
\end{remark}

\subsection{Application}
\label{sec:application}

Using the main theorem (\ref{MTIntr}), the above Seiberg\textendash Witten invariant
conjecture~(\ref{CONJ}) may be transformed into an additivity property of
analytic invariants $h^1(\cL)$.  In order to state it, we need the following
notation.

  For a fixed vertex \(v\) of the graph \(\Gamma\),
  let \(\Gamma_i\) be the components of \(\Gamma \setminus v\), and let
${\widetilde{X}}_i$ be a small tubular neighbourhood of $E_i \coloneqq \bigcup_{w \in
  \cV(\Gamma_i)} E_w$ in $\widetilde{X}$.  Let $(X_i,o)$ be the normal surface
singularity (with dual resolution graph $\Gamma_i$) obtained by collapsing the
curve $E_i \subset {\widetilde{X}}_i$ to a point.

\begin{theorem}
  \label{AN}
  Consider a family of singularities which satisfy the next property: for any
  non-rational $(X,o)$ in the family, there exists at least one vertex $v$
  (called \emph{splitting} vertex) in its (minimal) resolution graph \(\Gamma\)
  such that all the singularities $(X_i,o)$ are in the family.

  Then, for such a family, the validity of (\ref{CONJ}) 
for all the members of the family is equivalent to the next additivity
  property: every non-rational singularity \((X,o)\) in the family has a
splitting  vertex \(v\) satisfying:
  \begin{equation}\label{eq:6}
    h^1(\cO_{\widetilde{X}}(l'))= \cH_{\sigma,v}^{\mathid{pol}}(1) + 
    \sum_i h^1(\cO_{{\widetilde{X}}_i}(R_i(l'))) \quad \text{for \(l' \in R\)},
  \end{equation}
where $R_i$ is the natural cohomological 
restriction defined in (\ref{MT})(\ref{item:3}).
\end{theorem}
Note that the above additivity property \eqref{eq:6}
does not involve any part of Seiberg\textendash Witten theory. 

\begin{remark}\label{rem:extend}
  For fixed \((X,o)\) and \(v\),
  the validity of \eqref{eq:6}  for all \(l' \in R\) implies
  its validity for all \(l' \in R + \sum_{w \neq v} \setZ_{{}\geq0} E_w\).

  The reason is that
  \([l'] = [l' + \sum_{w \neq v} m_w E_w]\) and
  \([R_i(l')] = [R_i(l' + \sum_{w \neq v} m_w E_w])\)
  for any integers \(m_w\),
  hence Remark~\ref{RIntr}(\ref{RIntr2}) and Equation~\eqref{AK} applies to show
  the desired implication.
\end{remark}

For splice-quotient singularities, the additivity formula \eqref{eq:6}
was proved by T. Okuma in \cite{Opg}.  In fact, Okuma's formula gave
the idea of the existence of the set of purely topological
identities \eqref{eq:1}, and was the starting point of our
investigation.

As an application, we verify
Conjecture (\ref{CONJ}) for splice-quotients.  These singularities were
introduced recently by Neumann and Wahl \cite{NWuj2,NWuj}.
Since their definition  is rather
involved, we omit it.  The interested
reader may consult \cite{NWuj2,NWuj,Opg}.

Splice-quotients include rational and minimal elliptic
singularities (see \cite{Ouac-c}), and also the singularities
which admit a good $\setC^*$ action.  For splice-quotient
singularities and for the canonical \Spinc{} structure, the
conjecture was verified in \cite{NO1,NO2} (for some sporadic
cases, see also \cite{Stev}).  Here, as a byproduct, 
 we get the general case:
\begin{corollary}\label{TH}
  Conjecture (\ref{CONJ}) is true for any splice-quotient singularity.
\end{corollary}

Theorem~(\ref{AN}) and Corollary~(\ref{TH}) are proved in §\ref{Sec:6}.

\section{Preliminaries and notations.}\label{Sec:2}

\subsection{Notations regarding the plumbing representation}\label{NOT}

In the sequel we fix a negative definite tree $\Gamma$ as in §\ref{Sec:1}.  Notice that
$L'$ can be identified with the dual lattice of $L$.
 It is generated by the elements $E^*_w$,
where $(E^*_w, E_u)=\delta_{wu}$ is the Kronecker delta function.
The matrix \(I\) of the inclusion $L\hookrightarrow L'$ in the basis
${\{E_w\}}_w$ of \(L\) and the basis ${\{E^*_w\}}_w$ of \(L'\) is exactly 
the matrix of the intersection form in the basis \({\{E_w\}}_{w}\), namely,
$I_{ww}=b_w$ for all $w$, and for $u \neq w$,  we have $I_{uw}=1$ if
$u$ and $w$ are adjacent, and \(I_{uw} = 0\) otherwise.

By duality, \(L' \cong H_2(\widetilde{X}(\Gamma), \Sigma; \setZ)\), 
and  $L'/L \cong H_1(\Sigma; \setZ)$.  We denote the latter group
by $H$.  Let $\lvert H \rvert $ and $\widehat{H}$ denote its order and Pontrjagin dual
$\Hom(H,\setC^*)$, respectively.  Sometimes we write $d=\det(\Gamma)$ for
$\det(-I)=\lvert H \rvert $.
We define 
\begin{equation}
  \label{eq:30}
  a_{u w} \coloneqq -{\vert H \rvert} \cdot ( E^*_u, E^*_w ) = -{\lvert H \rvert}
  \cdot {(I^{-1})}_{u w}.
\end{equation}
Notice that every $a_{u w}$ is a \emph{positive} integer.

For any $u \in \cV(\Gamma)$ we write $\delta_u$ for the degree of $u$ in $\Gamma$ and we set:
\begin{align}
  \label{eq:16}
  \alpha_u &\coloneqq \sum_{w\in\cV(\Gamma)} (\delta_w - 2) a_{u w}, \\
  \label{eq:18}
  \beta_u &\coloneqq \sum_{w\in\cV(\Gamma)} (\delta_w - 2) a_{u w}^2.
\end{align}

Next we consider some topological/combinatorial invariants
of  $\Sigma$ and $\Gamma$.

\subsection{The Casson\textendash Walker invariant.}
\label{CW}

Let $\lambda(\Sigma)$ denote the Casson\textendash Walker invariant of $\Sigma$, normalized as in
\cite[(4.7)]{Lescop}.  Then
from
\cite{Ratiu} one has:
\begin{equation}\label{eq:15}
  -24 \frac{\lambda(\Sigma)}{\lvert H \rvert } = \sum_{w\in\cV(\Gamma)} b_w + 3 \lvert \cV(\Gamma) \rvert +
  \frac{1}{\lvert H \rvert} \cdot \sum_{w \in \cV(\Gamma)} (\delta_w - 2) a_{w w}.
\end{equation}

\subsection{\Spinc{} structures.}\label{Char}

As it is well-known, see e.g.\ \cite[(2.4.16)]{GS}, the set of \Spinc{}
structures is an \(H^2\) torsor for any manifold admitting a \Spinc{}
structure.  Let \(*\)  denote the action of \(H^2\) on the set of \Spinc{}
structures.
Recall that for any \(h \in H^2\) and \Spinc{} structure \(\sigma\), the action and the
Chern class interact as
\(c_1(h * \sigma) = c_1(\sigma) + 2h\).

For our plumbed manifold \(\widetilde{X}(\Gamma)\), there is a
canonical \Spinc{} structure \(\widetilde{\sigma_{\mathid{can}}}\), whose Chern class
is characterized by (see \cite[2.7\textendash2.9]{nemethi02._seiber_witten}) 
\begin{equation}
  \label{eq:4}
  (c_1(\widetilde{\sigma_{\mathid{can}}}), E_w) = b_w + 2 \quad \text{for all \(w \in \cV(\Gamma)\)}.
\end{equation}
Hence, there is a bijection between \(L'\) and the set of \Spinc{}
structures of \(\widetilde{X}(\Gamma)\) which assigns \(l' \in L'\) to \(l' *
\widetilde{\sigma_{\mathid{can}}}\).

Similarly, the set of \Spinc{} structures of the boundary \(\Sigma\) is an \(H\) torsor.
The restriction of \Spinc{} structures commute with the action via the
canonical map \(L' \to H\).
Since this homomorphism is surjective, every \Spinc{} structure of \(\Sigma\) extends
to  \(\widetilde{X}(\Gamma)\).

By definition,  the \emph{canonical} \Spinc{} structure
$\sigma_{\mathid{can}}$ on $\Sigma$ is the restriction of the canonical \Spinc{} structure
$\widetilde{\sigma_{\mathid{can}}}$ of \(\widetilde{X}(\Gamma)\).

\subsection{The Reidemeister\textendash Turaev torsion and the Seiberg\textendash Witten
  invariant.}
\label{RT}

  For any $\sigma \in
\Spinc(\Sigma)$, we consider the \emph{Reidemeister\textendash Turaev torsion}
 $\et_\sigma = \sum_{h\in H} \et_{\sigma}(h) h \in {\setQ}[H]$ from
\cite{Tu5}.  We will write $\et_\sigma(\Sigma)$ for
$\et_\sigma(0)$.  Then, by \cite[Theorem~2.4]{Nico5}, the \emph{Seiberg\textendash Witten
invariant} $\ssw_\sigma(\Sigma)$ of $\Sigma$ associated with $\sigma \in
\Spinc(\Sigma)$ equals (note our sign convention):
\begin{equation}\label{RT1}
  \ssw_\sigma(\Sigma) = \frac{\lambda(\Sigma)}{\lvert H \rvert} - \et_{\sigma}(\Sigma).
\end{equation}
By \cite[3.8, 5.7]
{nemethi02._seiber_witten}, $\et_{\sigma}(\Sigma)$ can be determined from the
graph $\Gamma$ via Fourier transform as follows.

First, for any $\rho \in \widehat{H}$ and fixed vertex $u\in\cV(\Gamma)$, we define
a rational function in $t$:
\begin{equation}\label{RT2}
  P_{\rho,u}(t) \coloneqq \prod_{w\in\cV(\Gamma)}  {(1-\rho([E^*_w]) t^{a_{w u}})}^{\delta_w-2},
\end{equation}
 where $[E^*_w]$
is the class of $E^*_w$ in $H = L'/L$.  Take also $h_\sigma \in H$
such that $h_\sigma * \sigma_{\mathid{can}} = \sigma$.
Next, for any non-trivial character $\rho \in \widehat{H} \setminus
\{1\}$, find a vertex $u_\rho \in \cV(\Gamma)$ such that either
$\rho([E^*_{u_\rho}])\neq1$, or $u_\rho$  has an adjacent vertex 
$u$ with $\rho([E^*_{u}])\neq1$.  Then
the Fourier transform of $\et$ is
\begin{equation}\label{RT3}
  \widehat{\et_{\sigma}}(\rho) = {\rho(h_{\sigma})}^{-1} 
\cdot \lim_{t\to 1} P_{\rho,u_{\rho}}(t) \quad (\rho \neq 1).
\end{equation}
In the sequel, this limit will be denoted simply by
$P_{\rho,u_{\rho}}(1)$.
Recall that $\widehat{\et_{\sigma}}(1) = 0$.  Therefore:
\begin{equation}\label{eq:11}
  \et_\sigma(\Sigma) = \frac{1}{\lvert H \rvert} \cdot \sum_{\rho \in \widehat{H} \setminus \{1\}}
  \widehat{\et_{\sigma}}(\rho).
\end{equation}
 If $\lvert H \rvert = 1$
then $\et_\sigma(\Sigma)=0$ for the unique \Spinc{} structure $\sigma$, hence
$\ssw_\sigma(\Sigma)=\lambda(\Sigma)$.

\subsection{The rational function $\cH_{\sigma, u}(t)$}
\label{Hilbert}

For any $\sigma \in \Spinc(\Sigma)$ and  $u \in \cV(\Gamma)$ one defines 
\begin{equation*}
\cH_{\sigma, u}(t) \coloneqq \frac{1}{\lvert H \rvert} \cdot \sum_{\rho \in
\widehat{H}} {\rho(h_{\sigma})}^{-1} \cdot P_{\rho,u}(t),
\quad \text{where \(h_\sigma*\sigma_{\mathid{can}}=\sigma\)}.
\end{equation*}

\subsection{Invariants associated with the distinguished vertex $v$.}
\label{sec:invar-dist-vertex}

Recall that for a fixed vertex \(v\) of \(\Gamma\), the components of \(\Gamma
\setminus v\) are the graphs \(\Gamma_i\).
Let $v_i$ denote the unique vertex of $\Gamma_i$ which is adjacent to $v$
in $\Gamma$.

We indicate by a subscript \(i\) when we use invariants of \(\Gamma_i\) instead of
\(\Gamma\).
For example, we write \(d_i = \det \Gamma_i\), \(H_i = H_1(\Sigma_i; \setZ)\), \(L_i\),
\(a_{u w, i}\) and so on.

We regard $L_i$ as a sublattice of $L$ via the natural inclusion
\(H_2(\widetilde{X}(\Gamma_i); \setZ) \hookrightarrow H_2(\widetilde{X}(\Gamma); \setZ)\).
Hence, for any $w \in \cV(\Gamma_i)$, we have $E_{w, i} = E_w$.

\begin{definition}
  \label{MT}\mbox{}
  \begin{enumerate}
  \item\label{item:2} Consider the setup of §\ref{Sec:1}.  For a \Spinc{}
    structure \(\sigma\) of \(\Sigma\), its \emph{restriction} \(\sigma_i\) to \(\Sigma_i\) is
    defined to be the
    restriction of any extension \(\widetilde{\sigma} \in
    \Spinc(\widetilde{X}(\Gamma))\) of \(\sigma\) satisfying~\eqref{eq:2}
    to the submanifold \(\Sigma_i\).
    In other words,
    $\widetilde{\sigma}=l'*\widetilde{\sigma_{\mathid{can}}}$ for some $l'\in
    L'$ with $[l']*\sigma_{\mathid{can}}=\sigma$ and 
    \begin{equation}
      \label{eq:5}
      -1 < \left(l', E^*_v \right) \leq 0.
    \end{equation}
  \item\label{item:3} 
The  \emph{restriction} $R_i\colon L'\to L_i'$ is the homomorphism induced by
the inclusion  $\widetilde{X}(\Gamma_i)\hookrightarrow \widetilde{X}(\Gamma)$  on second cohomology groups.
In other words,
$R_i(E^*_w) = E^*_{w,i}$ if $w \in \cV(\Gamma_i)$, and $R_i(E^*_w) = 0$ otherwise.
Therefore, for
$l' = \sum_{w} r_w E_w = \sum_{w} s_w E^*_w$, one has 
\begin{equation}\label{MT2}
  R_i(l') = \sum_{w \in \cV(\Gamma_i)} s_w E^*_{w,i} = r_vE^*_{v_i,i} +
  \sum_{w \in \cV(\Gamma_i)} r_w E_w.
\end{equation}
  \end{enumerate}
\end{definition}
Since $R_i(l')$ is characterized by $( R_i(l'), E_{w}) = 
(l', E_{w})$ for all $w \in \cV(\Gamma_i)$,  the last equality in \eqref{MT2} follows.
One can verify that $\sigma_i \in \Spinc(\Sigma_i)$ 
is independent of the choice of \(\widetilde{\sigma}\) thanks to \eqref{eq:5}.
Since the canonical \Spinc{} structure of \(\widetilde{X}(\Gamma)\) restricts 
to the canonical \Spinc{} structure of \(\widetilde{X}(\Gamma_i)\), the 
restriction of the canonical \Spinc{} structure of \(\Sigma\) to \(\Sigma_i\) 
is the canonical one.  Moreover,
the restriction of \(\sigma = [l'] * \sigma_{\mathid{can}}\) is \([R_i(l')] * 
\sigma_{\mathid{can},i}\)
provided that \(r_v \coloneqq (l',E_v^*) \in (-1,0]\).
The number \(r_v\) depends only on \(\sigma\) and not on the 
choice of \(l'\).

\subsection{Pseudo-characters.}\label{PCH}

We will need to extend the expression \eqref{RT2} for 
an arbitrary map $\psi\colon \cV(\Gamma) \to \setC^*$  by 
\begin{equation}\label{eq:99}
  P_{\psi,v}(t) \coloneqq \prod_{w\in\cV(\Gamma)} {(1 - \psi(w) t^{a_{w v}})}^{\delta_w-2}.
\end{equation}
For such a map $\psi$ and vertex $w \in \cV(\Gamma)$, we define
\begin{equation*}
  \df_w(\psi) \coloneqq {\psi(w)}^{b_w} \prod_{j=1}^{\delta_w} \psi(w(j)),
\end{equation*}
where ${\{w(j)\}}_j$ are the vertices of $\Gamma$ adjacent to $w$.
The map $\psi$ is
called a \emph{pseudo-character} (associated with the vertex $v$)
if $\df_w(\psi)=1$ for all
$w \neq v$.  Their collection will be denoted by $\widetilde{H}$.  We set
$\df(\psi) \coloneqq \df_v(\psi)$.  Notice that pseudo-characters $\psi$
with $\df(\psi)=1$ are exactly the characters of $H$ via the
correspondence $\psi(w)=\psi([E^*_w])$.
In fact,  $\psi$ can be regarded as a character on $L'$ (which does not
necessarily descend to $H$): any $\psi \in \widetilde{H}$ gives a
morphism  $L' \to \setC^*$ defined by
\begin{equation*}
  \psi \left( \sum_w m_w E^*_w \right) \coloneqq \prod_w {\psi(w)}^{m_w}.
\end{equation*}

\subsection{Notations regarding rational functions.}\label{FN}

\begin{enumerate}
\item\label{FN1}
We write any rational function $R$ as $R^{\mathid{pol}} + R^{<0}$, where
$R^{\mathid{pol}}$ is a polynomial and $R^{<0}$ is a rational function with
negative degree.  For $R$ without pole at $0$ 
we shall refine it further: one writes $R^{<0}$
in a unique way as a finite sum
\begin{align*}
  R^{<0}(t) &= \sum_{\alpha\neq0}(\Laurentpartat{\alpha} R)(t),\\
  \text{where}\quad
  (\Laurentpartat{\alpha} R)(t) &= \sum_{k>0}\frac{a_{\alpha,k}}
{{(1 - \alpha t)}^k},
  \quad  (\alpha \in \setC^*,  a_{\alpha,k} \in \setC).
\end{align*}
\item\label{FN2}
For any rational function  $R(t)$ with Laurent expansion
$\sum_{k\geq k_0} a_k {(t - 1)}^k$ at $t=1$, we write $\constanttermatone{R(t)}$ for the
coefficient $a_0$.  Notice that if \(1\) is not a pole of $R$ then
$\constanttermatone{R(t)} = R(1)$.
\end{enumerate}

The next identities  are elementary and their  proofs
are left to the reader.
\begin{lemma}\label{ELEM}
For any  $0 \leq q < d$ one has
\begin{align}
  \label{ELEM1}
  \frac{1}{d} \sum_{\alpha^d=\epsilon} \frac{\alpha^{-q}}{1-\alpha t} &= \frac{t^q}{1-\epsilon t^d}, \\
  \label{eq:17}
  \constanttermatone{\left( \frac{1}{d} \sum_{\alpha^d=1} \frac{\alpha^{-q}}{1 - \alpha t}
    \right)} &= \frac{d-1-2q}{2d}, \\
  \label{ELEM2}
  \frac{1}{d} \sum_{\alpha^d=1} \frac{\alpha^{-q}}{{(1-\alpha t)}^2} &=
  \frac{dt^q}{{(1- t^d)}^2} - \frac{(d-q-1)t^q}{1-t^d}, \\
  \label{ELEM3}
  \constanttermatone{\left( \frac{1}{d} \sum_{\alpha^d=1} \frac{\alpha^{-q}}{{(1 - \alpha t)}^2}
    \right)} &= - \frac{(d-1)(d-5)}{12d} - \frac{q^2+2q-qd}{2d}.
\end{align}
\end{lemma}

\section{Identities about determinants and restrictions.}\label{Sec:3}

Our calculation will extensively use the following general properties of
graph-determinants.

\begin{lemma}\label{Fact}
  \begin{enumerate}[(a)]
  \item\label{Fact0} Consider two vertices \(u,w \in \cV(\Gamma)\) of \(\Gamma\).  Let \(\Gamma \setminus
    \overline{uw}\) be the subgraph of \(\Gamma\) obtained by deleting the
    path connecting \(u\) and \(w\) (including \(u\) and \(w\)).  Then
    \begin{equation}\label{eq:100}
      a_{uw} = \det (\Gamma \setminus \overline{uw}).
    \end{equation}
  \item\label{Facta} For every component \(\Gamma_i\) of \(\Gamma \setminus v\) and vertex \(w\)
    of \(\Gamma_i\)
    \begin{equation}
      \label{eq:31}
      a_{v w} = a_{v_i w, i} \cdot \det (\Gamma \setminus v \setminus \Gamma_i).
    \end{equation}

  \item\label{Factb} For any \(u\in\cV(\Gamma)\) one has
    \begin{equation}\label{eq:200}
      a_{u u} \cdot \prod_{w\in\cV(\Gamma)} a_{u w}^{\delta_w - 2} = 1.
    \end{equation}
  \item\label{Factc} Consider a decomposition of \(\Gamma\) as follows:

\begin{center}
  \begin{picture}(0,0)%
\includegraphics{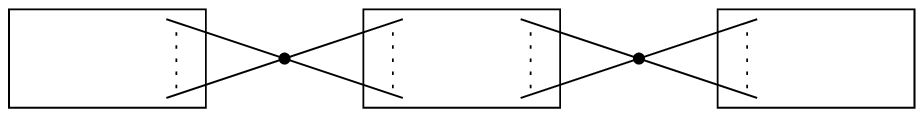}%
\end{picture}%
\setlength{\unitlength}{4144sp}%
\begingroup\makeatletter\ifx\SetFigFontNFSS\undefined%
\gdef\SetFigFontNFSS#1#2#3#4#5{%
  \reset@font\fontsize{#1}{#2pt}%
  \fontfamily{#3}\fontseries{#4}\fontshape{#5}%
  \selectfont}%
\fi\endgroup%
\begin{picture}(4842,508)(211,-2672)
\put(1261,-2446){\makebox(0,0)[lb]{\smash{{\SetFigFontNFSS{8}{9.6}{\familydefault}{\mddefault}{\updefault}{\color[rgb]{0,0,0}$G'$}%
}}}}
\put(2926,-2446){\makebox(0,0)[lb]{\smash{{\SetFigFontNFSS{8}{9.6}{\familydefault}{\mddefault}{\updefault}{\color[rgb]{0,0,0}$G$}%
}}}}
\put(4546,-2446){\makebox(0,0)[lb]{\smash{{\SetFigFontNFSS{8}{9.6}{\familydefault}{\mddefault}{\updefault}{\color[rgb]{0,0,0}$G''$}%
}}}}
\put(226,-2446){\makebox(0,0)[lb]{\smash{{\SetFigFontNFSS{8}{9.6}{\familydefault}{\mddefault}{\updefault}{\color[rgb]{0,0,0}$\Gamma$:}%
}}}}
\put(3736,-2626){\makebox(0,0)[lb]{\smash{{\SetFigFontNFSS{8}{9.6}{\familydefault}{\mddefault}{\updefault}{\color[rgb]{0,0,0}$u$}%
}}}}
\put(2116,-2626){\makebox(0,0)[lb]{\smash{{\SetFigFontNFSS{8}{9.6}{\familydefault}{\mddefault}{\updefault}{\color[rgb]{0,0,0}$v$}%
}}}}
\end{picture}%

\end{center}
Above, the subgraphs \(G'\), \(G\) and \(G''\) can be empty.
If \(G\) is empty then \(v\) and \(u\) is connected by a single edge.
The vertices \(v\) and \(u\) are not allowed to be the same.

Then  (with the convention  $\det(\emptyset)=1$), one  has:
\begin{equation}
  \label{eq:28}
  \begin{split}
    \det(\Gamma) \cdot \det(G) ={}& \det(G \cup G' \cup v) \cdot \det(G \cup G''\cup u) \\
    &- \det(G') \cdot \det(G'') \cdot {\det (G \setminus \overline{uv})}^2.
  \end{split}
\end{equation}
(Here $G \cup G'\cup v$ and $G \cup G''\cup u$ also contain the edges adjacent to $v$ and $u$,
respectively.)
\end{enumerate}
\end{lemma}

\begin{proof}
Equation~\eqref{eq:100} is proved in \cite[(20.2)]{MR817982}.
Equation~\eqref{eq:31} follows from \eqref{eq:100} and by noting that the
determinant of graphs is multiplicative over disjoint union of graphs.

Statement (\ref{Factb}) immediately follows from
\eqref{eq:100} and \eqref{eq:200}
by an easy induction on the number of vertices of the graph.

The claim (\ref{Factc}) is an exercise on graph determinants.
For example,
let us consider the components of \(G\),
which are connected only to \(v\) and not to \(u\).
By moving these components from \(G\) to \(G'\),
we reduce to the case
that \(v\) and \(G\) are connected by a single edge.
Similarly, we reduce to the case when
\(u\) and \(G\) are also connected by a single edge.
Then \eqref{eq:28} follows from \cite[Lemma~12.7]{NWuj2}.
\end{proof}

\begin{corollary}\label{Coruj}
  Using the decomposition of (\ref{Fact})(\ref{Factc}), for any \(S \subseteq
  \cV(G'')\), one has:
  \begin{equation}
    \label{eq:33}
    \begin{split}
     {\left( \prod_{w \notin S} a_{w v}^{\delta_w - 2}\right)}^{-1} &= 
     \det G' \cdot \det (G \cup G'' \cup u) \\
     &\cdot {(\det G' \cdot \det (G \setminus \overline{u v}))}^{\sum_{w \in S} \delta_w - 2} \cdot
     \prod_{w \in S} {\det (G'' \setminus \overline{u w})}^{\delta_w - 2}.
    \end{split}
  \end{equation}
The subgraph \(G\) is  allowed to be empty.
Furthermore, \(v\) and \(u\) are allowed to be the same, and in this case \(G\) is
empty and one should write \(G''\) instead of \(G \cup G'' \cup u\) in the 
formula.
  In particular,
  \begin{equation}
    \label{eq:34}
    \prod_{w \in \cV(\Gamma) \setminus \cV(\Gamma_i)} a_{w v}^{\delta_w-2} = \frac{1}{d_i}.
  \end{equation}
\end{corollary}

\begin{proof}
The left hand side of the first equation, by \eqref{eq:200}, 
is   $a_{v v}  \prod_{w \in S} a_{w v}^{\delta_w - 2}$, which equals the right hand
side by
\eqref{eq:100}.
The second equation follows from the first one by the choices \(u
\coloneqq v\), \(G'' \coloneqq \Gamma_i\), \(S \coloneqq \cV(G'')\)
and \(G' = \bigcup_{j \neq i} \Gamma_j\).
Note that \(\sum_{w \in S} (\delta_w - 2) = -1\) 
and \(\prod_{w \in S} {\det (G'' \setminus
  \overline{u w})}^{\delta_w - 2} = 1\) 
(the latter is \eqref{eq:200} applied to \(G'' \cup u\)).
\end{proof}

\begin{lemma}\label{DetRel}
   For any \(x \in L'\) and its restrictions \(x_i \coloneqq R_i(x)\)
    (see~\ref{MT}(\ref{item:3}))
  \begin{align}
  \label{eq:7}
  x - \sum_i x_i &= - \frac{d (x, E^*_v)}{a_{v v}} E^*_v \\
  \label{eq:8}
  x^2 - \sum_i x_i^2 &= - \frac{d {(x, E^*_v)}^2}{a_{v v}}.
  \end{align}
\end{lemma}

\begin{proof}
The main idea of the proof of \eqref{eq:7} is that
since the scalar product is definite,
it is enough to verify that
the scalar product with either side of the equation agree,
at least on a basis of \(L'\) over \(\setQ\).
We choose the basis consisting of the \(E_w\) for \(w \neq v\) and \(E^*_v\).
It is easy to verify that
the scalar product of either side of \eqref{eq:7} with \(E_w\) is
\(0\) for \(w \neq v\),
and the scalar product of either side with \(E_v^*\) is \((x, E^*_v)\).

Equation~\eqref{eq:8} is the scalar product of \eqref{eq:7} with \(x\).
Here we use the identity \((x, x_i) = x_i^2\), which is true,
since \(x_i\) is the restriction of \(x\).
\end{proof}

\section{Additivity formulas. Proof of Theorem (\ref{MTIntr}).}\label{Sec:4}

We break the main identity \eqref{eq:2} into the additivity formulas
\eqref{AK} and \eqref{eq:36}, and we also break the latter one into \eqref{AL}
and \eqref{AT}.

\begin{proposition}\label{prop:2}
  With the notations of §\ref{Sec:2}, (especially of (\ref{MT})), one has:
  \begin{align}
    \label{AK}
    {c_1(\widetilde{\sigma})}^2+ \lvert \cV(\Gamma) \rvert - \sum_i \left( {c_1(\widetilde{\sigma}_i)}^2+
      \lvert \cV(\Gamma_i) \rvert \right) &= 1
    -\frac{{(\alpha_v + d + 2 d r_v)}^2}{d a_{v v}}, \\
    \label{eq:36}
    \ssw_\sigma(\Sigma) - \sum_i \ssw_{\sigma_i}(\Sigma_i) &= - \cH^{\mathid{pol}}_{\sigma,v}(1)
    - \frac{1}{8} + \frac{{(\alpha_v + d + 2 d r_v)}^2}{8 d a_{v v}}, \\
    \label{AL}
    24 \frac{\lambda}{\lvert H \rvert } - \sum_i 24 \frac{\lambda_i}{\lvert
      H_i \rvert} &= -3 + \frac{d^2 - \beta_v}{d a_{v v}}, \\
    \label{AT}
    \et_\sigma(\Sigma) - \sum_i \et_{\sigma_i}(\Sigma_i) &= \cH^{\mathid{pol}}_{\sigma,v}(1) +
    \frac{d^2 - \beta_v}{24 d a_{v v}} -
    \frac{{(\alpha_v + d + 2 d r_v)}^2}{8 d a_{v v}}.
  \end{align}
\end{proposition}
Equation~\eqref{eq:36} is a combination of~\eqref{AL}, \eqref{AT} and
\eqref{RT1}.  The proof of \eqref{AT} is given in §\ref{Sec:5}.  Here we prove
\eqref{AK} and \eqref{AL} as applications of \eqref{eq:8}.

\begin{proof}[Proof of \eqref{AK}]
We apply~\eqref{eq:8} to \(x\coloneqq c_1(\widetilde{\sigma})\).  Then
\(x_i=c_1(\widetilde{\sigma}_i)\), and 
\begin{equation}
  \label{eq:9}
  {c_1(\widetilde{\sigma})}^2 - \sum_i {c_1(\widetilde{\sigma}_i)}^2 = - \frac{d
    {(c_1(\widetilde{\sigma}), E^*_v)}^2}{a_{v v}}.
\end{equation}
By the definition of \(r_v\):
\begin{equation}
  \label{eq:10}
  2r_v = (c_1(\widetilde{\sigma}) - c_1(\widetilde{\sigma_{\mathid{can}}}), E^*_v).
\end{equation}
Next, we compute \((c_1(\widetilde{\sigma_{\mathid{can}}}), E^*_v)\).
Expressing the Chern class from \eqref{eq:4} as
\begin{equation*}
   c_1(\widetilde{\sigma_{\mathid{can}}}) = \sum_w E_w - \sum_w (\delta_w - 2) E^*_w,
\end{equation*}
and then using  \eqref{eq:16} we get
\begin{equation}
  \label{eq:12}
   (c_1(\widetilde{\sigma_{\mathid{can}}}), E^*_v) = 1 + \frac{\alpha_v}{d}.
\end{equation}
Finally, combining \eqref{eq:9}, \eqref{eq:10} and \eqref{eq:12} gives the desired
formula.
\end{proof}

\begin{proof}[Proof of \eqref{AL}]
This time, we apply~\eqref{eq:8} first to \(x\coloneqq E^*_w\) for some $w \neq
v$.  Then \(x_i = E^*_{w,i}\) if \(w \in \cV(\Gamma_i)\), and \(x_i=0\) otherwise.
Hence, \eqref{eq:8} reads as
\begin{equation}
  \label{eq:13}
  -\frac{a_{w w}}{d} + \frac{a_{w w,i}}{d_i} = -
  \frac{a_{v w}^2}{d a_{v v}}, \quad w \in \cV(\Gamma_i).
\end{equation}
Next, we apply~\eqref{eq:8} to \(x \coloneqq E_v\).  Then \(x_i =
E^*_{v_i,i}\) and we get:
\begin{equation}
  \label{eq:14}
  b_v + \sum_i \frac{a_{v_i v_i, i}}{d_i} = - \frac{d}{a_{v v}}.
\end{equation}
The claimed equality is a linear combination of \eqref{eq:18}, \eqref{eq:13}, \eqref{eq:14}
and \eqref{eq:15}, where the latter is applied to \(\Sigma\) and all the \(\Sigma_i\).
\end{proof}

\section{Proof of \protect\eqref{AT}.}\label{Sec:5}

\subsection{Breaking up the torsion}
\label{sec:breaking-up-torsion}

We start with some preparations.  For an arbitrary map $\psi\colon \cV(\Gamma)\to \setC^*$ we define
\begin{align*}
  {\cV(\Gamma)}_\psi &\coloneqq \{w\in\cV(\Gamma) : \psi(w)=1\},\\
  \supp(\psi) &\coloneqq \cV(\Gamma)\setminus {\cV(\Gamma)}_\psi,\\
  \psi_i &\coloneqq {\left. \psi \right|}_{\cV(\Gamma_i)}\colon \cV(\Gamma_i)\to \setC^*.
\end{align*}
\begin{lemma}
  \label{lem:1}
  Let \(\Gamma\) be a negative definite tree and \(\psi\colon \cV(\Gamma) \to \setC^*\) a function on it.
  Then the least degree term of the Laurent series of \(P_{\psi,v}\) 
(see \eqref{eq:99}) at \(1\) is
  \begin{equation}
    \label{eq:19}
    P_{\psi,v}(t) = \prod_{w  \notin {\cV(\Gamma)}_\psi} {\left( 1 - \psi (w) \right)}^{\delta_{w} - 2}
    \cdot \prod_{w \in {\cV(\Gamma)}_\psi} a_{v w}^{\delta_{w} - 2} \cdot
    {\left( 1 - t \right)}^{n} + O({(1 - t)}^{n+1}),
  \end{equation}
  where
  \begin{equation}
    \label{eq:20}
    \begin{split}
      n &\coloneqq \sum_{w \in {\cV(\Gamma)}_\psi} (\delta_w - 2) \\
      &= 
      -2 \lvert \{ \text{components of \({\cV(\Gamma)}_\psi\)} \} \rvert
      + \lvert \{ \text{edges going out of \({\cV(\Gamma)}_\psi\)} \} \rvert.
    \end{split}
  \end{equation}
  In particular, if every component of \({\cV(\Gamma)}_\psi\) has a vertex with
  at least two outgoing edges (e.g.\ \(\psi\) is a non-trivial character) then
  \(n \geq 0\) with equality if and only if all components have exactly two
  outgoing edges.
\end{lemma}

\begin{proof}
This is mainly a repetition of \cite[A.7]{nemethi02._seiber_witten}.
The first formula obviously follows from~\eqref{eq:99} by taking the least
degree term in \(t-1\) of every factor of the product.
This gives \(\sum_{w \in {\cV(\Gamma)}_\psi} (\delta_w - 2)\) for the degree \(n\) of the least degree
term.
The second equality of~\eqref{eq:20} is a well-known
identity for circuit-free graphs.
\end{proof}

\begin{proposition}\label{et}
  For all non-trivial character $\rho \in \widehat{H}$ and \Spinc{} structure $\sigma =
  h * \sigma_{\mathid{can}}$ of \(\Sigma\) with \(h \in H\)
  \begin{equation}
    \label{eq:23}
    \frac{1}{d} \widehat{\et_{\sigma}}(\rho) = \frac{1}{d} {\rho(h)}^{-1} \cdot P_{\rho,v}(1) +
    \begin{cases}
      \frac{1}{d_i} \widehat{\et_{\sigma_i}}(\rho_i) & \text{if \({\left. \rho
        \right|}_{(\cV(\Gamma) \setminus \cV(\Gamma_i)) \cup \{v_i\}} = 1\)} \\
      0 & \text{otherwise},
    \end{cases}
  \end{equation}
  where \(\sigma_i\) is the restriction of \(\sigma\) defined in Definition
  (\ref{MT})(\ref{item:2}).
\end{proposition}

\begin{proof}
Obviously, if \(\rho(v)=1\) then \(\rho_i \coloneqq {\left. \rho \right|}_{\cV(\Gamma_i)}\) is a
character of \(H_i\).

The proof of the proposition is a case-by-case verification.

First, let us consider the case when \(\rho\) is non-trivial at \(v\) or one of its neighbours.
Then we can choose \(u_\rho \coloneqq v\) in \eqref{RT3}, so
\eqref{eq:23} immediately follows.

In the remaining cases, \(\rho\) is trivial on \(v\) and its 
neighbours \(v_i\).    By the second part of Lemma~\ref{lem:1},
all three terms of~\eqref{eq:23} are \(0\) (because \(n > 0\)) unless every
component of \({\cV(\Gamma)}_\rho\) has exactly two outgoing edges.
Hence the only remaining case is when every component of 
\({\cV(\Gamma)}_\rho\) has exactly two outgoing edges.  Therefore 
$\sum_{w \notin {\cV(\Gamma)}_\rho} (\delta_w - 2) = -2$, 
and there exists an index $i$ with $\supp(\psi) \subset \cV(\Gamma_i)$.  Hence,
the upper case of Equation~\eqref{eq:23} should hold.

Let \(u_\rho\) be the vertex
of the component \({\cV(\Gamma)}_\rho(v)\)
of \({\cV(\Gamma)}_\rho\) containing \(v\) where its two outgoing edges start.

We decompose \(\Gamma\) into subgraphs as shown in
the next picture.

\begin{center}
  \begin{picture}(0,0)%
\includegraphics{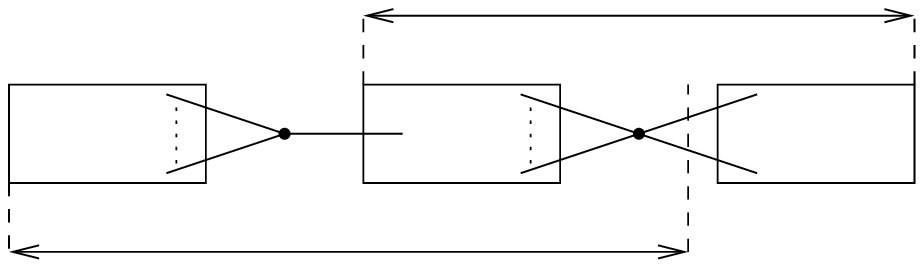}%
\end{picture}%
\setlength{\unitlength}{4144sp}%
\begingroup\makeatletter\ifx\SetFigFontNFSS\undefined%
\gdef\SetFigFontNFSS#1#2#3#4#5{%
  \reset@font\fontsize{#1}{#2pt}%
  \fontfamily{#3}\fontseries{#4}\fontshape{#5}%
  \selectfont}%
\fi\endgroup%
\begin{picture}(4842,1606)(211,-3221)
\put(1261,-2446){\makebox(0,0)[lb]{\smash{{\SetFigFontNFSS{8}{9.6}{\familydefault}{\mddefault}{\updefault}{\color[rgb]{0,0,0}$G'$}%
}}}}
\put(2926,-2446){\makebox(0,0)[lb]{\smash{{\SetFigFontNFSS{8}{9.6}{\familydefault}{\mddefault}{\updefault}{\color[rgb]{0,0,0}$G$}%
}}}}
\put(226,-2446){\makebox(0,0)[lb]{\smash{{\SetFigFontNFSS{8}{9.6}{\familydefault}{\mddefault}{\updefault}{\color[rgb]{0,0,0}$\Gamma$:}%
}}}}
\put(2116,-2626){\makebox(0,0)[lb]{\smash{{\SetFigFontNFSS{8}{9.6}{\familydefault}{\mddefault}{\updefault}{\color[rgb]{0,0,0}$v$}%
}}}}
\put(3691,-2626){\makebox(0,0)[lb]{\smash{{\SetFigFontNFSS{8}{9.6}{\familydefault}{\mddefault}{\updefault}{\color[rgb]{0,0,0}$u_\rho$}%
}}}}
\put(4546,-2446){\makebox(0,0)[lb]{\smash{{\SetFigFontNFSS{8}{9.6}{\familydefault}{\mddefault}{\updefault}{\color[rgb]{0,0,0}$G''$}%
}}}}
\put(3691,-1726){\makebox(0,0)[lb]{\smash{{\SetFigFontNFSS{8}{9.6}{\familydefault}{\mddefault}{\updefault}{\color[rgb]{0,0,0}$\Gamma_i$}%
}}}}
\put(2296,-3166){\makebox(0,0)[lb]{\smash{{\SetFigFontNFSS{8}{9.6}{\familydefault}{\mddefault}{\updefault}{\color[rgb]{0,0,0}$\cV(\Gamma)_\rho(v)$}%
}}}}
\end{picture}%

\end{center}

We express the terms of~\eqref{eq:23} in terms of determinants of
subgraphs using \eqref{RT2}, \eqref{RT3} and \eqref{eq:33}:
\begin{align*}
  P_{\rho,v}(1) 
 &= \frac{\det(G') \cdot {\left( \det(G \setminus \overline{v u_\rho})
      \right)}^{2}}{\det(G \cup G'' \cup u_\rho)} \prod_{w \notin {\cV(\Gamma)}_\rho} {\left(
      \frac{(1-\rho([E^*_w]))}{\det(G'' \setminus \overline{u_\rho w})} \right)}^{\delta_w-2}, \\
   \widehat{\et_{\sigma}}(\rho) =
P_{\rho, u_\rho}(1) &= \frac{\det(G' \cup G \cup v)}{\det(G'')}
 \prod_{w \notin {\cV(\Gamma)}_\rho} {\left(
      \frac{(1-\rho([E^*_w]))}{\det(G'' \setminus \overline{u_\rho w})} \right)}^{\delta_w-2}, \\
\widehat{\et_{\sigma_i}}(\rho_i) =
 P_{\rho_i, u_\rho}(1) &= \frac{\det(G)}{\det(G'')}
 \prod_{w \notin {\cV(\Gamma)}_\rho} {\left(
      \frac{(1-\rho([E^*_w]))}{\det(G'' \setminus \overline{u_\rho w})} \right)}^{\delta_w-2}.
\end{align*}
Note that
 \(\rho(h) = \rho_i(h_i)\) where \(\sigma_i = h_i * \sigma_{\mathid{can},i}\)  by
Definition~\ref{MT}, and hence these factor out of \eqref{eq:23}.
We can also factor out the \(\prod_{w \notin {\cV(\Gamma)}_\rho}\) product.
Finally, recall that
 $d_i = \det(G \cup G''\cup u_\rho)$ and \(d = \det \Gamma\).
Hence~\eqref{eq:23} reduces to (\ref{Fact})(\ref{Factc}).
\end{proof}

\subsection{Principal part of the Hilbert function}
\label{sec:princ-part-hilb}

 Next, we concentrate on $\cH_{\sigma,v}$.  We invite the reader to recall
the notations from (\ref{PCH})\textendash(\ref{FN}).

\begin{lemma}\label{Psi}
  For every non-trivial pseudo-character \(\psi\) associated with \(v\)
  \begin{equation}
    \label{eq:35}
    \Laurentpartat{1} P_{\psi, v} (t) =
    \begin{cases}
      \frac{1}{d_i} \cdot \frac{P_{\psi_i, v_i}(1) \cdot (1 - \psi_i(v_i))}{1
        - t} & \text{if \(\supp \psi \subseteq \cV(\Gamma_i)\) and \(\psi(v_i) \neq 1\)}, \\
      0 & \text{for all other \(\psi \neq 1\)}.
    \end{cases}
  \end{equation}
\end{lemma}

\begin{proof}
We apply Lemma~\ref{lem:1}.   By the pseudo-character relations, all 
components of ${\cV(\Gamma)}_\psi$  have at least two outgoing edges except possibly the 
component containing \(v\), which can have only one outgoing edge, which must
start at \(v\).
Hence the lower degree of the Laurent expansion of \(P_{\psi, v}\) at \(1\) is at
least \(-1\) with equality if and only if all the components have the minimum
number of outgoing edges declared above.
In particular, if \(\Laurentpartat{1} P_{\psi, v} \neq 0\) then \(\psi(v_i) \neq 1\) for
some \(i\) and \(\supp(\psi) \subset \cV(\Gamma_i)\).
This proves the lower case
of~\eqref{eq:35}.

To prove the upper case, note that by~\eqref{eq:31} for any $w\in\cV(\Gamma_i)$
\begin{equation*}
  P_{\psi,v}(t) = P_{\psi_i, v_i} (t^{\det (\Gamma \setminus v \setminus \Gamma_i)}) \cdot
  (1 - \psi_i(v_i) t^{a_{v v_i}}) \cdot \prod_{w \notin \cV(\Gamma_i)} {(1 - t^{a_{w v}})}^{\delta_w - 2}.
\end{equation*}
Obviously, $\psi_i$ is a
non-trivial character of $H_i$, hence $P_{\psi_i,v_i}$ is regular
at $1$.
 Moreover,
 $\sum_{w \notin \cV(\Gamma_i)} (\delta_w - 2) = -1$.  Thus
\begin{equation*}
  \Laurentpartat{1} P_{\psi,v}(t) = \frac{1}{1 - t} \cdot P_{\psi_i,v_i}(1) \cdot
  (1 - \psi_i(v_i))\cdot \prod_{w\notin\cV(\Gamma_i)} a_{w v}^{\delta_w - 2}.
\end{equation*}
For the last product one can use \eqref{eq:34}, and this finishes the proof.
\end{proof}

We fix an \(l' \in L'\) with \(\sigma = [l'] * \sigma_{\mathid{can}}\) 
and \(-1 <  r_v=(l',E^*_v) \leq 0\) and \(\sigma_i = [l'_i] *
\sigma_{\mathid{can}, i}\) for the restriction \(l'_i\) of \(l'\) to \(\Gamma_i\).
Note that all the poles $\alpha$  of \(P_{\rho, v}\) are roots of unity.
\begin{equation*}
  d \cdot \cH^{<0}_{\sigma,v}(t) =
  \sum_\alpha \sum_{\rho \in \widehat{H}} {\rho([l'])}^{-1} \Laurentpartat{\alpha} P_{\rho,v}(t) = \sum_\alpha
  \sum_{\substack{\psi \in \widetilde{H} \\ \df(\psi)=\alpha^{d}}}
  {\psi(l')}^{-1} \alpha^{dr_v} (\Laurentpartat{1} P_{\psi,v})(\alpha t),
\end{equation*}
where the last equality is obtained via the substitutions \(\psi(w) \coloneqq
\rho(w) \alpha^{- a_{v w}}\) implying \(\psi(x) = \rho([x]) \alpha^{d (x, E^*_v)}\) for all \(x \in
L'\).  To compute \(\df_w(\psi)\), we have used the identity \(I \cdot I^{-1} = 1\)
in the form
\begin{equation*}
  b_w a_{w v} + \sum_i a_{w v_i} =
  \begin{cases}
    -d & \text{if \(w = v\)} \\
     0 & \text{if \(w \neq v\)}.
 \end{cases}
\end{equation*}

To compute further, we apply Lemma~(\ref{Psi}) to index the pseudo-characters
\(\psi\) for which the summand maybe non-zero by characters \(\psi_i\) of \(H_i\)
with \(\psi_i(v_i) \neq 1\):
\begin{multline*}
  d \cdot \cH^{<0}_{\sigma,v}(t) =
\sum_{\alpha^d=1} \alpha^{dr_v}(\Laurentpartat{1} P_{1,v})(\alpha t) \\ + \sum_\alpha \sum_i
  \frac{1}{d_i} \sum_{\substack{\psi_i \in \widehat{H_i} \\ \psi_i(v_i) = \alpha^d \neq 1}}
 {\psi_i([l'_i])}^{-1} \alpha^{dr_v} P_{\psi_i,v_i}(1) (1 - \psi_i(v_i)) \frac{1}{1 - \alpha t}.
\end{multline*}
Using~\eqref{RT3} in the form
$\widehat{\et_{\sigma_i}}(\psi_i) = {\psi_i([l'_i])}^{-1} P_{\psi_i,v_i}(1)$, and summing in
the variable
\(\alpha\) by \eqref{ELEM1} (recall that \(-d < d r_v \leq 0\)):
\begin{equation}\label{Cont2}
  \constanttermatone{\cH^{<0}_{\sigma,v}(t)} = \constanttermatone{\left(
      \frac{1}{d} \sum_{\alpha^d=1} \alpha^{dr_v} (\Laurentpartat{1}P_{1,v})(\alpha t) \right)}
  + \sum_i \frac{1}{d_i} \sum_{\substack{\psi_i \in \widehat{H_i} \\ \psi_i(v_i)\neq1}}
  \widehat{\et_{\sigma_i}}(\psi_i).
\end{equation}

\subsection{Additivity formula for torsion}
\label{sec:addit-form-tors}

Now, we are ready to establish an additivity formula for the torsion.  By
\eqref{et} and \eqref{eq:11} 
\begin{equation*}
  \et_\sigma(\Sigma) = \constanttermatone{\cH_{\sigma,v}(t)} - \frac{1}{d}
  \constanttermatone{P_{1,v}(t)} + 
  \sum_i \frac{1}{d_i} \sum_{\substack{\psi_i \in \widehat{H_i} \setminus 1 \\ \psi_i(v_i)=1}}
  \widehat{\et_{\sigma_i}}(\psi_i).
\end{equation*}
Then, using $\cH_{\sigma,v} = \cH_{\sigma,v}^{\mathid{pol}} + \cH_{\sigma,v}^{<0}$ and
\eqref{Cont2} we get the next identity.  We highlight it, since it shows the
more conceptual source of the correction constant in \eqref{AT}:
\begin{equation}\label{tors}
  \et_\sigma(\Sigma) - \sum_i \et_{\sigma_i}(\Sigma_i) = \cH_{\sigma,v}^{\mathid{pol}}(1) + \frac{1}{d}
  \constanttermatone{\left( \sum_{\alpha^d=1} \alpha^{dr_v}
      (\Laurentpartat{1}P_{1,v})(\alpha t) - P_{1,v}(t) \right)}.
\end{equation}

The last two terms depend only on the coefficients of terms
with non-positive degree of the
 Laurent expansion of
$P_{1,v}$ at $1$.  These terms
can be computed  elementarily:
\begin{multline*}
  P_{1,v}(t) = \prod_w {(1-t^{a_{v w}})}^{\delta_w-2} \\
  =
  \frac{1}{a_{v v}} \left(
    \frac{1}{{(t-1)}^2} + \frac{1 + \alpha_v / 2}{t - 1} + \frac{{(\alpha_v + 1)}^2}{8}
    + \frac{\beta_v - 1}{24} + O(t-1) \right).
\end{multline*}
 Hence (\ref{ELEM}) and a simple
computation provides \eqref{AT}.

\section{Proof of Theorem (\ref{AN}) and Corollary (\ref{TH}).}\label{Sec:6}

In this section we combine our surgery formula with the main result of Okuma
from \cite{Opg} to derive the results of §\ref{Back}. 

Okuma's article \cite{Opg} 
uses a constant invariant of the Taylor expansion at the origin of
$R$ in place of our $R^{\mathid{pol}}(1)$.
This constant invariant was later called the periodic constant, which
terminology we adopt.

In the first paragraphs we prove that they are equal.
After the proof appeared in a public preprint of this article,
the result (Lemma~\ref{LPC}) was also incorporated into Okuma's article
as Proposition 4.8.

\begin{definition}[{Periodic constant \cite[3.9]{NO1}, \cite[just before
    Proposition 4.8]{Opg}}]
  \label{PC}
  Let  $F(t) = \sum_{i\geq 0} a_i t^i$  be a formal power series.  
  Suppose that for some positive integer $p$,
  the expression $\sum_{i=0}^{pn-1} a_i$ is a polynomial $P_p(n)$ in the
  variable $n$.  Then the constant term of $P_p(n)$ is independent of $p$.
  We call this constant term the \emph{periodic constant} of $F$ and
  denote it by $\pc{F}$.
\end{definition}

For rational functions, one has the following equivalent 
description of the periodic constant.  Here,
we identify the rational function $R$ with its Taylor expansion at the 
origin.

\begin{lemma}\label{LPC}
  Let $R$ be a rational function having poles only at infinity and roots of
  unity.  Then $R$ has a periodic constant and $\pc{R} =
  R^{\mathid{pol}}(1)$, where $R^{\mathid{pol}}$ is the polynomial part of $R$
  as in (\ref{FN})(\ref{FN1}).
\end{lemma}
\begin{proof} Write
\begin{equation*}
  R(t) = R^{\mathid{pol}}(t) + \sum_{\substack{k \geq 0 \\ 0 \leq j < p}}
  a_{k j}\frac{t^j}{{(1-t^p)}^{k+1}} \quad (a_{k j} \in \setC),
\end{equation*}
where the sum is finite.  Note that if two formal power series $F_1$ and
$F_2$ have periodic constants then $\pc{(F_1+F_2)} = \pc{F_1} + \pc{F_2}$.  Also,
every polynomial $A$ has a periodic constant, namely, $\pc{A} = A(1)$.  Hence it
is enough to prove that
$t^j {(1-t^p)}^{-(k+1)}= \sum_{l\geq 0} \binom{k+l}{k} t^{lp+j}$
 admits a periodic  constant, which is \(0\).  Indeed, the constant term of 
$\sum_{l=0}^{n-1} \binom{k+l}{k} = \binom{k+n}{k+1}$ 
is \(0\) as a polynomial in $n$.
\end{proof}

\begin{proof}[Proof of Theorem (\ref{AN})]
We prove the statement by induction on the number of vertices in the
dual resolution graph of the singularity \((X,o)\).

First, let us suppose that the class satisfies the Seiberg\textendash Witten invariant
conjecture~\ref{CONJ}.  Then expressing the Seiberg\textendash Witten invariants from
\eqref{eq:3} and substituting the result into \eqref{eq:1}, we obtain
\eqref{eq:6} (for all singularity \((X,o)\) in the class and all splitting vertex \(v\)).

To prove the converse, let us assume that the class satisfies \eqref{eq:6}.  We prove
the Seiberg\textendash Witten invariant conjecture (\ref{CONJ}) for every member of the
family by induction on the number of vertices of the dual resolution graph.
For rational singularities, (\ref{CONJ}) is true by \cite[Theorem~6.2]{Line}.
This starts the induction.

For a non-rational \((X,o)\) in the class, let us choose a vertex \(v\) of the
dual resolution graph satisfying \eqref{eq:6}.  Let \(l' \in R\).  Then
\(R_i(l') \in R(\Gamma_i) + L{(\Gamma_i)}_{e}\) 
by \eqref{MT2}, since the first term
$r_v E^*_{v_i,i}$ of its right hand side is a non-negative rational cycle, 
 and its second term is contained in \(R(\Gamma_i)\).  So, by
the induction hypothesis and Remark~(\ref{rem:extend}), Equation~\eqref{eq:3}
applies to \((X_i,o)\) and \(R_i(l')\).
Combining these with \eqref{eq:1} and \eqref{eq:6} for \((X,o)\) and \(v\), we
obtain \eqref{eq:3} for \((X,o)\).
\end{proof}

\begin{proof}[Proof of Corollary (\ref{TH})]
The corollary follows from Theorem (\ref{AN}) by Okuma's results from
\cite{Opg}, which show that the class of splice-quotient singularities satisfy
all the necessary conditions.

Specifically, for every splice-quotient singularity \((X,o)\) and vertex \(v\)
of the dual resolution graph, the singularities \((X_i,o)\) are also
splice-quotient by \cite[2.16]{Opg}.

Moreover, the additivity formula~\eqref{eq:6} for all \(l' \in R\) 
and $v$ with degree at least \(3\) is a
combination of \cite[Theorem 4.5 and Lemma 4.2(3)]{Opg} and Lemma (\ref{LPC}).
\end{proof}

\section{Examples}\label{Sec:7}

\subsection{$\mathbf{\Sigma=S^3_{-d}(K)}$}
\label{Knot}

Let $K\subset S^3$ be an algebraic knot, i.e.\ the link of an analytic irreducible
plane curve singularity $f\colon (\setC^2,0)\to (\setC,0)$.  Let $\mu$ and $\Delta(t)$ be its Milnor
number and Alexander polynomial, respectively.  Let $\Sigma \coloneqq S^3_{-d}(K)$
be obtained by $(-d)$-surgery ($d \in \setN^+$) along $K \subset S^3$.  The Heegaard Floer
homology of $\Sigma$ was computed in \cite{nemethi05._heegaar_floer_s_k} in
terms of $\Delta$ (see also
\cite[Theorem~4.1]{OSZuj}).  Here we recover the formula \cite[4.3]{nemethi05._heegaar_floer_s_k} for $\ssw_*(\Sigma)$
from our results.

Let the (minimal) good resolution of $(\setC^2,f^{-1}(0))$ be given by the
schematic diagram

\begin{center}
  \begin{picture}(0,0)%
\includegraphics{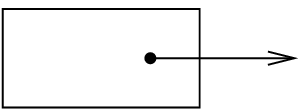}%
\end{picture}%
\setlength{\unitlength}{4144sp}%
\begingroup\makeatletter\ifx\SetFigFontNFSS\undefined%
\gdef\SetFigFontNFSS#1#2#3#4#5{%
  \reset@font\fontsize{#1}{#2pt}%
  \fontfamily{#3}\fontseries{#4}\fontshape{#5}%
  \selectfont}%
\fi\endgroup%
\begin{picture}(1467,474)(889,-2638)
\put(1261,-2446){\makebox(0,0)[lb]{\smash{{\SetFigFontNFSS{8}{9.6}{\familydefault}{\mddefault}{\updefault}{\color[rgb]{0,0,0}$\Gamma_1$}%
}}}}
\put(1531,-2581){\makebox(0,0)[lb]{\smash{{\SetFigFontNFSS{8}{9.6}{\familydefault}{\mddefault}{\updefault}{\color[rgb]{0,0,0}$v_1$}%
}}}}
\put(2341,-2446){\makebox(0,0)[lb]{\smash{{\SetFigFontNFSS{8}{9.6}{\familydefault}{\mddefault}{\updefault}{\color[rgb]{0,0,0}$K$}%
}}}}
\end{picture}%

\end{center}

Write $m_f$ for the vanishing order of the lifting of $f$ along the
exceptional divisor $E_{v_1}$.  Then (see \cite{nemethi05._heegaar_floer_s_k}), a possible plumbing
graph of $\Sigma$ is

\begin{center}
  \begin{picture}(0,0)%
\includegraphics{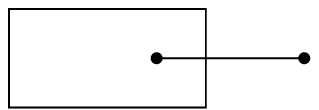}%
\end{picture}%
\setlength{\unitlength}{4144sp}%
\begingroup\makeatletter\ifx\SetFigFontNFSS\undefined%
\gdef\SetFigFontNFSS#1#2#3#4#5{%
  \reset@font\fontsize{#1}{#2pt}%
  \fontfamily{#3}\fontseries{#4}\fontshape{#5}%
  \selectfont}%
\fi\endgroup%
\begin{picture}(2071,474)(211,-2638)
\put(226,-2446){\makebox(0,0)[lb]{\smash{{\SetFigFontNFSS{8}{9.6}{\familydefault}{\mddefault}{\updefault}{\color[rgb]{0,0,0}$\Gamma$:}%
}}}}
\put(1261,-2446){\makebox(0,0)[lb]{\smash{{\SetFigFontNFSS{8}{9.6}{\familydefault}{\mddefault}{\updefault}{\color[rgb]{0,0,0}$\Gamma_1$}%
}}}}
\put(1531,-2581){\makebox(0,0)[lb]{\smash{{\SetFigFontNFSS{8}{9.6}{\familydefault}{\mddefault}{\updefault}{\color[rgb]{0,0,0}$v_1$}%
}}}}
\put(2206,-2581){\makebox(0,0)[lb]{\smash{{\SetFigFontNFSS{8}{9.6}{\familydefault}{\mddefault}{\updefault}{\color[rgb]{0,0,0}$v$}%
}}}}
\put(1981,-2311){\makebox(0,0)[lb]{\smash{{\SetFigFontNFSS{8}{9.6}{\familydefault}{\mddefault}{\updefault}{\color[rgb]{0,0,0}$-d - m_f$}%
}}}}
\end{picture}%

\end{center}

Let $v$ be the `new' vertex.  Then $\Gamma \setminus v$ has only one component, namely
$\Gamma_1$, which can be blown down completely, hence $\Sigma_1=S^3$.  One can verify
that $H=\setZ_d $ and it is generated by $[E^*_v]$.  Hence  $\widehat{H}$ consists
of the maps \(\rho\) given by $\rho([k E^*_v])=\xi^k $ for all $d$th roots of unity
$\xi$.  Moreover, the \Spinc{} structures of $\Sigma$ are $[q E^*_v] *
\sigma_{\mathid{can}}$ for $0 \leq q < d$.  
Then, using e.g.\ the formula \cite[11.3]{MR817982} for $\Delta$, one has
\begin{equation*}
\cH_{[qE^*_v] * \sigma_{\mathid{can}},v}(t) = \frac{1}{d} \sum_{\xi^d=1} \xi^{-q}
  \frac{\Delta(\xi t)}{{(1 - \xi t)}^2}.
\end{equation*}
One can write $\Delta(t) = 1 + (t-1)\mu / 2 + {(t-1)}^2 \sum_l a_l t^l$.  Hence
\begin{equation*}
  \cH^{\mathid{pol}}_{[qE^*_v]*\sigma_{\mathid{can}},v}(t) = \frac{1}{d} \sum_\xi \xi^{-q}
  \sum_l a_l \xi^l t^l = \sum_l a_{q+ld} t^{q+ld}.
\end{equation*}
Note that $a_{v v} = 1$, hence $(q E^*_v, E^*_v) = - q / d \in (-1,0]$ and so
$r_v = - q / d$.  Recall e.g.\ from \cite[11.1]{MR817982} that $\mu - 2 = \alpha_v$.  Thus,
using \eqref{eq:36}, we recover \cite[4.3]{nemethi05._heegaar_floer_s_k} as promised:
\begin{equation*}
  \ssw_{[qE^*_v] * \sigma_{\mathid{can}}}(S^3_{-d}(K)) = - \sum_l a_{q+ld} + \frac{{(\mu
    - 2 + d - 2 q)}^2}{8d} - \frac{1}{8}.
\end{equation*}
Similarly (with slightly more computations) one can recover the
Seiberg\textendash Witten invariant of $S^3_{-p/q}(K)$, too (here \(p/q \in \setQ\),
 \(p/q > 0\)); for a possible formula see \cite[4.5]{nemethi._heegaar_floer_s_k}.

\subsection{Seifert manifolds.}\label{star}

Let $\Sigma$ be a Seifert manifold.  Recall that either $\Sigma$ or $-\Sigma$ can be
realized as a negative definite plumbing (and $\ssw(-\Sigma)=-\ssw(\Sigma)$), hence we
may assume without loss of generality that $\Sigma=\Sigma(\Gamma)$ for a (minimal)
negative definite graph $\Gamma$. 
We will assume that $\Gamma$ is not a string (i.e.\ $\Sigma$ is not
a lens space).  Then $\Gamma$ is star-shaped; let $v$ be its
central vertex. 
There exists an affine complex surface singularity $X$ whose
link at the origin is $\Sigma$, and which admits a good $\setC^*$ action. In
particular, its affine coordinate ring $A$ is graded.

First we show how
$\cH_{\sigma,v}(t)$ and its periodic constant can be expressed from 
the Seifert invariants of $\Sigma$. 

Let ${(\alpha_i,\omega_i)}_{i=1}^r$ denote the
normalized Seifert invariants of $\Sigma$
(for more details,  see~\cite{nemethi04._seiber_witten}). 
 Set $\alpha=\operatorname{lcm}(\alpha_i : i=1, \dots, r)$ and
$o = \alpha \cdot \lvert H \rvert / \prod_i \alpha_i$.  We denote the
end-vertices (i.e.\ vertices of degree $1$) by ${\{w_i\}}_i$.
Then $[E^*_v]$ and ${\{[E^*_{w_i}]\}}_i$ generate $H$, hence $l'\in
L'$ can be written as $l' = a E^*_v + \sum_i a_i E^*_{w_i}$  modulo $L$.
Set $\widetilde{a} \coloneqq \alpha (a + \sum_i a_i / \alpha_i)$.  Then, 
by \cite[Theorem (3.1)]{nemethi04._seiber_witten}, for 
$\sigma=[l']*\sigma_{can}$  one has
\begin{equation}\label{eq:8.2.1}
  \cH_{\sigma,v}(t) = \sum_{l \geq -\widetilde{a}/o} \max\left(0, 1 + a - lb_v +
    \sum_{i=1}^r \left\lfloor \frac{-l\omega_i + a_i}{\alpha_i} \right\rfloor\right)
  t^{ol+\widetilde{a}}.
\end{equation}
In the case $\sigma=\sigma_{can}$, one has $a = a_i = \widetilde{a} = 0$.
Moreover, we claim that
\begin{equation}\label{eq:8.2.2}
  \pc{\cH_{\mathid{can}, v}}=\sum_{l\geq 0} \max\left(0, -1 + lb_v - \sum_{i=1}^r \left\lfloor \frac{-l \omega_i}{\alpha_i}
    \right\rfloor\right).
\end{equation}
The idea of the proof is the following: let us define the polynomial
\begin{equation*}
P(t) \coloneqq \sum_{l\geq 0} \max\left(0, -1 + lb_v - \sum_{i=1}^r \left\lfloor \frac{-l \omega_i}{\alpha_i}
    \right\rfloor\right) t^{ol}.
\end{equation*}
By the identity $\max(0,x)-\max(0,-x)=x$ we get that
\begin{equation*}
  \cH_{\mathid{can},v}(t)-P(t) = \sum_{l \geq 0} \left( 1 - lb_v +
    \sum_{i=1}^r \left\lfloor \frac{-l\omega_i}{\alpha_i} \right\rfloor\right)
  t^{ol}.
\end{equation*}
Then a computation shows that the periodic constant of the last expression is
zero.  Hence
$\pc{\cH_{\mathid{can}, v}} = P(1)$, which is exactly \eqref{eq:8.2.2}.

Note that by \cite{Opg,Pi1} the right hand side of (\ref{eq:8.2.1})
is the Hilbert (Poincaré) series of a graded
$A$-module.  If $\sigma=\sigma_{can}$ then this module is 
exactly  $A$.  On the other hand, by \cite{Pi1,Dolg}, 
the expression from the right hand side of (\ref{eq:8.2.2}) 
is exactly the \emph{geometric genus} $p_g$ of $(X,o)$.
In particular, we have also proved that the periodic constant of the
Poincaré series of the graded algebra $A$ is exactly the geometric genus of
the singularity. 

Now, let us apply (\ref{MTIntr}) for  $\sigma=\sigma_{can}$.
Since all the components of \(\Gamma \setminus v\) are strings, 
they support rational singularities.  Therefore, by
\cite[4.1.1]{Line}, 
\begin{equation*}
  \ssw_{\mathid{can}}(\Sigma_i) + \frac{{c_1(\widetilde{\sigma_{\mathid{can}, i}})}^2 +
  \lvert \cV(\Gamma_i) \rvert}{8} = 0.
\end{equation*}
Hence  (\ref{MTIntr}) reads as $\ssw_{\mathid{can}}(\Sigma) +
({c_1(\widetilde{\sigma_{\mathid{can}}})}^2 + \lvert \cV(\Gamma) \rvert) / 8 = - p_g$. 

Notice that this is exactly the claim of the 
Seiberg\textendash Witten invariant conjecture~(\ref{CONJ})
for weighted homogeneous singularities and for the canonical \Spinc{}
structure.  Its original proof from
\cite{nemethi04._seiber_witten} is based on completely different
combinatorial identities. 

We would like to emphasize that, in general, $\pc{\cH}$ can be a rather
complicated arithmetical expression.  E.g., when $\Sigma$ is the Seifert
\(3\)-manifold $\Sigma(a,b,c)$ (the link of $x^a+y^b+z^c$ with \(a\), \(b\), \(c\)
pairwise relative prime numbers), then $\pc{\cH_{\mathid{can},v}}$ is the number of
interior lattice points in the tetrahedron with vertices \((0,0,0)\),
\((a,0,0)\), \((0,b,0)\), \((0,0,c)\).
(This can be expressed by Dedekind sums by a result of Mordell.)
\bibliographystyle{amsplainurl}\bibliography{hivatkozas}

\makeatletter \def\@strippedMR{}
  \def\@scanforMR#1#2#3\endscan{\def\@strippedMR{#1#2#3}%
  \ifx#1M\ifx#2R\def\@strippedMR{#3}\fi\fi} \def\@MRgetnumber#1 #2\relax{\href%
  {http://www.ams.org/mathscinet-getitem?mr=#1}}
  \def\MR#1{\relax\ifhmode\unskip\spacefactor3000 \space\fi
  \@scanforMR#1\endscan \@firstofone{\expandafter\@MRgetnumber \@strippedMR}
  \relax{MR\@strippedMR}} \providecommand{\href}[2]{#2} \makeatother
  \def\cprime{$'$}
\providecommand{\bysame}{\leavevmode\hbox to3em{\hrulefill}\thinspace}
\providecommand{\MR}{\relax\ifhmode\unskip\space\fi MR }
\providecommand{\MRhref}[2]{%
  \href{http://www.ams.org/mathscinet-getitem?mr=#1}{#2}
}
\providecommand{\href}[2]{#2}
\begin{thebibliography}{10}

\bibitem{Co}
Olivier Collin, \emph{Equivariant {C}asson invariant for knots and the
  {N}eumann-{W}ahl formula}, Osaka J. Math. \textbf{37} (2000), no.~1, 57--71.
  \MR{MR1750270 (2001d:57014)}

\bibitem{CoS}
Olivier Collin and Nikolai Saveliev, \emph{A geometric proof of the
  {F}intushel-{S}tern formula}, Adv. Math. \textbf{147} (1999), no.~2,
  304--314. \MR{MR1734525 (2001b:57023)}

\bibitem{CoS2}
\bysame, \emph{Equivariant {C}asson invariants via gauge theory}, J. Reine
  Angew. Math. \textbf{541} (2001), 143--169. \MR{MR1876288 (2002k:57077)}

\bibitem{Dolg}
I.~V. Dolgachev, \emph{Automorphic forms and weighted homogeneous
  singularities}, Funkt. Anal. Jego. Prilozh. \textbf{9} (1975), 67--68,
  English translation in Funct. Anal. Appl., {{\bf 9}} (1975), 149--151.

\bibitem{MR817982}
David Eisenbud and Walter Neumann, \emph{Three-dimensional link theory and
  invariants of plane curve singularities}, Annals of Mathematics Studies, vol.
  110, Princeton University Press, Princeton, NJ, 1985. \MR{MR817982
  (87g:57007)}

\bibitem{FS}
Ronald Fintushel and Ronald~J. Stern, \emph{Instanton homology of {S}eifert
  fibred homology three spheres}, Proc. London Math. Soc. (3) \textbf{61}
  (1990), no.~1, 109--137. \MR{MR1051101 (91k:57029)}

\bibitem{FMS}
Shinji Fukuhara, Yukio Matsumoto, and Koichi Sakamoto, \emph{{C}asson's
  invariant of {S}eifert homology {$3$}-spheres}, Math. Ann. \textbf{287}
  (1990), no.~2, 275--285. \MR{MR1054569 (91e:57027)}

\bibitem{GS}
Robert~E. Gompf and Andr{\'a}s~I. Stipsicz, \emph{{$4$}-manifolds and {K}irby
  calculus}, Graduate Studies in Mathematics, vol.~20, American Mathematical
  Society, Providence, RI, 1999. \MR{MR1707327 (2000h:57038)}

\bibitem{Lescop}
Christine Lescop, \emph{Global surgery formula for the {C}asson-{W}alker
  invariant}, Annals of Mathematics Studies, vol. 140, Princeton University
  Press, Princeton, NJ, 1996. \MR{MR1372947 (97c:57017)}

\bibitem{SI}
I.~Luengo-Velasco, A.~Melle-Hern{\'a}ndez, and A.~N{\'e}methi, \emph{Links and
  analytic invariants of superisolated singularities}, J. Algebraic Geom.
  \textbf{14} (2005), no.~3, 543--565. \MR{MR2129010 (2005m:32057)}

\bibitem{Line}
Andr{\'a}s N{\'e}methi, \emph{Line bundles associated with normal surface
  singularities}, part of \cite{Graded},
  \href{http://arxiv.org/abs/math.AG/0310084} {\path{arXiv:math.AG/0310084}}.

\bibitem{nemethi._heegaar_floer_s_k}
\bysame, \emph{On the {H}eegaard {F}loer homology of {$S^3_{-p/q}(K)$}},
  \href{http://arxiv.org/abs/math.GT/0410570} {\path{arXiv:math.GT/0410570}}.

\bibitem{Ninv}
\bysame, \emph{{\textquotedblleft}{W}eakly{\textquotedblright} elliptic
  {G}orenstein singularities of surfaces}, Invent. Math. \textbf{137} (1999),
  no.~1, 145--167. \MR{MR1703331 (2000e:32037)}

\bibitem{INV}
\bysame, \emph{Invariants of normal surface singularities}, Real and complex
  singularities, Contemp. Math., vol. 354, Amer. Math. Soc., Providence, RI,
  2004, pp.~161--208. \MR{MR2087811 (2005g:32040)}

\bibitem{nemethi05._heegaar_floer_s_k}
\bysame, \emph{On the {H}eegaard {F}loer homology of {$S\sp 3\sb {-d}(K)$} and
  unicuspidal rational plane curves}, Geometry and topology of manifolds,
  Fields Inst. Commun., vol.~47, Amer. Math. Soc., Providence, RI, 2005,
  pp.~219--234. \MR{MR2189934}

\bibitem{NOSZ}
\bysame, \emph{On the {O}zsv{\'a}th-{S}zab{\'o} invariant of negative definite
  plumbed 3-manifolds}, Geom. Topol. \textbf{9} (2005), 991--1042 (electronic),
  \href{http://arxiv.org/abs/math.GT/0310083} {\path{arXiv:math.GT/0310083}},
  \href{http://dx.doi.org/10.2140/gt.2005.9.991}
  {\path{doi:10.2140/gt.2005.9.991}}. \MR{MR2140997 (2006c:57011)}

\bibitem{Graded}
\bysame, \emph{Graded roots and singularities}, Singularities in geometry and
  topology, World Sci. Publ., Hackensack, NJ, 2007, pp.~394--463.
  \MR{MR2311495}

\bibitem{nemethi02._seiber_witten}
Andr{\'a}s N{\'e}methi and Liviu~I. Nicolaescu, \emph{{S}eiberg-{W}itten
  invariants and surface singularities}, Geom. Topol. \textbf{6} (2002),
  269--328 (electronic), \href{http://arxiv.org/abs/math.AG/0111298}
  {\path{arXiv:math.AG/0111298}},
  \href{http://dx.doi.org/10.2140/gt.2002.6.269}
  {\path{doi:10.2140/gt.2002.6.269}}. \MR{MR1914570 (2003i:14048)}

\bibitem{nemethi04._seiber_witten}
\bysame, \emph{{S}eiberg-{W}itten invariants and surface singularities. {II}.
  {S}ingularities with good {$\mathbb{C}\sp *$}-action}, J. London Math. Soc.
  (2) \textbf{69} (2004), no.~3, 593--607. \MR{MR2050035 (2005g:14070)}

\bibitem{nemethi05._seiber_witten}
\bysame, \emph{{S}eiberg-{W}itten invariants and surface singularities:
  splicings and cyclic covers}, Selecta Math. (N.S.) \textbf{11} (2005),
  no.~3-4, 399--451. \MR{MR2215260}

\bibitem{NO1}
Andr{\'a}s N{\'e}methi and Tomohiro Okuma, \emph{On the {C}asson invariant
  conjecture of {N}eumann-{W}ahl}, to appear in Journal of Algebraic Geometry,
  \href{http://arxiv.org/abs/math.AG/0610465} {\path{arXiv:math.AG/0610465}}.

\bibitem{NO2}
\bysame, \emph{The {S}eiberg{\textendash}{W}itten invariant conjecture for
  splice-quotients}, J. London Math. Soc. (2) \textbf{78} (2008), no.~1,
  143--154 (English), \href{http://dx.doi.org/10.1112/jlms/jdn020}
  {\path{doi:10.1112/jlms/jdn020}}.

\bibitem{NW}
Walter Neumann and Jonathan Wahl, \emph{{C}asson invariant of links of
  singularities}, Comment. Math. Helv. \textbf{65} (1990), no.~1, 58--78.
  \MR{MR1036128 (91c:57022)}

\bibitem{NWuj2}
Walter~D. Neumann and Jonathan Wahl, \emph{Complete intersection singularities
  of splice type as universal abelian covers}, Geom. Topol. \textbf{9} (2005),
  699--755 (electronic), \href{http://arxiv.org/abs/math/0407287}
  {\path{arXiv:math/0407287}}, \href{http://dx.doi.org/10.2140/gt.2005.9.699}
  {\path{doi:10.2140/gt.2005.9.699}}. \MR{MR2140991 (2006i:32037)}

\bibitem{NWuj}
\bysame, \emph{Complex surface singularities with integral homology sphere
  links}, Geom. Topol. \textbf{9} (2005), 757--811 (electronic),
  \href{http://arxiv.org/abs/math/0301165} {\path{arXiv:math/0301165}},
  \href{http://dx.doi.org/10.2140/gt.2005.9.757}
  {\path{doi:10.2140/gt.2005.9.757}}. \MR{MR2140992 (2006b:32042)}

\bibitem{Nico5}
Liviu~I. Nicolaescu, \emph{{S}eiberg-{W}itten invariants of rational homology
  3-spheres}, Commun. Contemp. Math. \textbf{6} (2004), no.~6, 833--866.
  \MR{MR2111431 (2005k:57031)}

\bibitem{Opg}
Tomohiro Okuma, \emph{The geometric genus of splice-quotient singularities}, to
  appear in {T}rans. {A}mer. {M}ath. {S}oc.,
  \href{http://arxiv.org/abs/math.AG/0610464} {\path{arXiv:math.AG/0610464}}.

\bibitem{Ouac-c}
\bysame, \emph{Universal abelian covers of certain surface singularities},
  Math. Ann. \textbf{334} (2006), no.~4, 753--773,
  \href{http://arxiv.org/abs/math.AG/0503733} {\path{arXiv:math.AG/0503733}}.
  \MR{MR2209255}

\bibitem{OSzP}
Peter Ozsv{\'a}th and Zolt{\'a}n Szab{\'o}, \emph{On the {F}loer homology of
  plumbed three-manifolds}, Geom. Topol. \textbf{7} (2003), 185--224
  (electronic), \href{http://dx.doi.org/10.2140/gt.2003.7.185}
  {\path{doi:10.2140/gt.2003.7.185}}. \MR{MR1988284 (2004h:57039)}

\bibitem{OSZ}
\bysame, \emph{On {H}eegaard diagrams and holomorphic disks}, European Congress
  of Mathematics, Eur. Math. Soc., Z{\"u}rich, 2005, pp.~769--781.
  \MR{MR2185780 (2006g:57061)}

\bibitem{OSZuj}
\bysame, \emph{Knot {F}loer homology and integer surgeries}, Algebr. Geom.
  Topol. \textbf{8} (2008), no.~1, 101--153 (English),
  \href{http://arxiv.org/abs/math.GT/0410300} {\path{arXiv:math.GT/0410300}}.

\bibitem{Pi1}
H.~Pinkham, \emph{Normal surface singularities with {$C\sp*$} action}, Math.
  Ann. \textbf{227} (1977), no.~2, 183--193. \MR{MR0432636 (55 \#5623)}

\bibitem{Ratiu}
A.~Ratiu, \emph{The {J}ones{\textendash}{W}itten invariants of tree manifolds},
  Ph.D. thesis, Universite Paris, July 1996.

\bibitem{Stev}
Jan Stevens, \emph{Universal abelian covers of superisolated singularities},
  2006, \href{http://arxiv.org/abs/math.AG/0601669}
  {\path{arXiv:math.AG/0601669}}.

\bibitem{Tu5}
Vladimir Turaev, \emph{Torsion invariants of {${\rm Spin}\sp c$}-structures on
  {$3$}-manifolds}, Math. Res. Lett. \textbf{4} (1997), no.~5, 679--695.
  \MR{MR1484699 (98k:57038)}

\end{thebibliography}
\end{document}